\documentclass[11pt,a4paper]{amsart}

\usepackage[T1]{fontenc}
\usepackage[utf8]{inputenc}
\usepackage{lmodern}
\usepackage{microtype}
\usepackage{amsmath,amssymb,mathtools,mathrsfs}
\usepackage[a4paper,margin=1in]{geometry}
\usepackage[hidelinks]{hyperref}
\usepackage[
  backend=biber,
  style=numeric,
  sortcites=true,
  sorting=nyt,
  doi=true,
  url=true,
  isbn=false
]{biblatex}
\hypersetup{
  pdftitle={Existence of bases implies the axiom of choice, a foundation-free proof},
  pdfsubject={Axiom of choice and vector space bases in ZF without Foundation},
  pdfkeywords={Axiom of choice, axiom of foundation, axiom of multiple choice,
    vector space bases, normal field extensions}
}

\allowdisplaybreaks[1]

\newcommand{\ZF}{\mathsf{ZF}}
\newcommand{\AC}{\mathsf{AC}}
\newcommand{\MC}{\mathsf{MC}}
\newcommand{\ZFm}{\mathsf{ZF}^{-}}
\newcommand{\Pow}{\mathcal P}
\newcommand{\Q}{\mathbb Q}
\newcommand{\Ecal}{\mathcal E}
\newcommand{\Ccal}{\mathcal C}
\newcommand{\Lcal}{\mathcal L}
\newcommand{\Var}{\mathsf{Var}}
\newcommand{\Symbols}{\mathsf{Symbols}}
\DeclareMathOperator{\Gal}{Gal}
\DeclareMathOperator{\Emb}{Emb}
\DeclareMathOperator{\Bas}{Bas}
\DeclareMathOperator{\coeff}{coeff}
\DeclareMathOperator{\dom}{dom}

\newtheorem{theorem}{Theorem}[section]

\newtheorem{overviewtheorem}{Theorem}

\newtheorem{lemma}[theorem]{Lemma}
\newtheorem{corollary}[theorem]{Corollary}
\theoremstyle{definition}
\newtheorem{definition}[theorem]{Definition}
\theoremstyle{remark}

\newtheorem{claim}{Claim}

\title[Vector space bases without Foundation]
{Existence of bases implies the axiom of choice, a foundation-free proof}

\author{Gabriel Fernandes}
\address{Department of Mathematics, Institute of Mathematics and Computer
Sciences, University of S\~ao Paulo, S\~ao Carlos, SP, Brazil}
\email{fernandes@icmc.usp.br}

\author{Renan Maneli Mezabarba}
\address{Department of Exact Sciences, State University of Santa Cruz,
Ilh\'eus, BA, Brazil}
\email{rmmezabarba@uesc.br}

\author{Vinicius de Oliveira Rodrigues}
\address{Department of Mathematics, Institute of Mathematics, Statistics and
Computer Science, University of S\~ao Paulo, S\~ao Paulo, SP, Brazil}
\email{vinior@ime.usp.br}

\date{}

\begin{document}

\begin{abstract}
We prove that, in Zermelo--Fraenkel set theory with the axiom of
Foundation removed, the statement that every vector space has a basis implies the Axiom of Choice ($\AC$), concluding that the classical equivalence between $\AC$ and the existence of bases does not require regularity.
This result extends to set theory with atoms.
\end{abstract}

\maketitle

\section{Introduction}

The existence of bases for vector spaces is one of the most familiar
applications of the Axiom of Choice ($\AC$), usually proved by Zorn's
lemma. Its early history includes Hamel's construction of a basis of
$\mathbb R$ over $\mathbb Q$, shortly after Zermelo's proof of the
well-ordering theorem~\cite{Hamel1905,Zermelo1904}, and Hausdorff's
treatment of abstract real linear spaces~\cite[p.~295]{Hausdorff1932}.
The case of arbitrary fields appears in Zorn's 1935 treatment of the
principle now known as Zorn's lemma~\cite{Zorn1935}. The converse
question, whether the existence of bases implies $\AC$, proved much
more difficult.

In Zermelo--Fraenkel set theory ($\ZF$), this converse was eventually
established by Blass~\cite{Blass}, by showing that the existence of
bases implies the Axiom of Multiple Choice ($\MC$). This axiom
permits the simultaneous selection of a nonempty finite subset from
each member of a family of nonempty sets. His proof then uses the
equivalence of $\MC$ and $\AC$ in $\ZF$~\cite{FelgnerJech1973}.
The two steps differ in their use of the Axiom of Foundation: the
first does not require it, whereas the second does.

To examine whether Foundation is necessary for the converse itself,
we work in Zermelo--Fraenkel set theory with the Axiom of Foundation
omitted ($\ZFm$). In this setting, the part of Blass's argument that
we shall use is the following.

\begin{theorem}[Blass {\cite{Blass}}]\label{thm:blass}
In $\ZFm$, if every vector space over a field of characteristic zero has a basis, then $\MC$ holds.
Thus,
\[
    \ZF\vdash \text{``every vector space has a basis''} \leftrightarrow \AC.
\]
\end{theorem}

The restriction to fields of characteristic zero is not stated in Blass' article, but his proof works under this weaker hypothesis. The
subsequent implication from $\MC$ to $\AC$ does require Foundation:
L\'evy's permutation-model construction establishes the relative
consistency of $\MC+\neg\AC$ with $\ZFm$~\cite{Levy1962}. Consequently,
Blass's argument left open whether the existence of bases implies
$\AC$ in $\ZFm$~\cites[p.~33]{Blass}{Karagila2017MO}[p.~423]{Morillon}[p.~121]{Philip2025}.

Is Foundation really necessary here? Given its usual role in
mathematics, an affirmative answer would be surprising.\footnote{The question
whether the existence of bases implies $\AC$ in $\ZFm$ is explicitly
recorded as open by Philip~\cite[\S7.2.2, p.~121]{Philip2025}.
Morillon~\cite[p.~423]{Morillon} explicitly records the corresponding
question in set theory with atoms.} Foundation
is generally regarded as dispensable for ordinary mathematics,
although it is useful in the metamathematics of set theory.\footnote{See
\cites[Chapter~6, p.~63]{Jech2003}[p.~47 and Chapter~III, \S4]{Kunen1980}[\S I.14]{Kunen2009}[p.~87]{FraenkelBarHillelLevy1973}.}
We show that it can indeed be omitted here. Retaining Blass's passage
to $\MC$, we use the hypothesis on bases again to construct
well-orderings directly, and thereby prove the following.

\begin{theorem}\label{thm:main}
In $\ZFm$, assume that every vector space over a field of characteristic
zero has a basis. Then $\AC$ holds.
Thus, in $\ZFm$, the following are equivalent:
\begin{enumerate}
    \item $\AC$;
    \item every vector space has a basis;
    \item every vector space over a field of characteristic zero has a basis.
\end{enumerate}
\end{theorem}

Section~\ref{sec:structure} explains this well-ordering argument by
presenting its main ingredients in reverse order of dependence. The
subsequent sections supply their proofs, including the required
facts from Galois theory, in $\ZFm$.

\section{Structure of the proof}\label{sec:structure}

We work throughout this section in $\ZFm$.

To prove Theorem~\ref{thm:main}, it suffices to show that every set can
be well-ordered. We shall place an arbitrary set inside an algebraic
field extension and obtain its well-ordering from the following
theorem. This is the first reduction of the proof: it turns the
set-theoretic question into a question about extending a well-ordering
from a field to a larger field.

\begin{overviewtheorem}\label{thm:algebraic}
Assume $\MC$. Let $L/K$ be a normal separable algebraic extension.
Assume that $\prec$ is a well-ordering of $K$ and that, for every well-orderable
field $F$ with $K\subseteq F\subseteq L$, the $F$-vector space $L$ has a basis.
Then $L$ is well-orderable by a well-ordering in which $(K, \prec)$ is an initial segment.
\end{overviewtheorem}

To apply this theorem, we need a suitable extension containing a copy
of the set we started with. The next theorem provides it using
$\MC$, which is already available from Blass's theorem.

\begin{overviewtheorem}\label{prop:embedding}
Assume $\MC$. For every set $X$ there exist a well-orderable field $K$
of characteristic zero, a normal separable algebraic extension $L/K$,
and an injection $X\to L$.
\end{overviewtheorem}

The construction uses L\'evy's partition lemma to write $X$ as an
ordinal-indexed union of disjoint finite sets. Replacing the elements
of $X$ by indeterminates gives a rational function field $L$. For each
finite set in the partition, we form the polynomial whose roots are its
indeterminates, and let $K$ be generated by the coefficients of these
polynomials. The coefficients have a well-ordered indexing, which
makes $K$ well-orderable, while the polynomials make $L/K$ algebraic,
normal and separable. The details are given in
Section~\ref{sec:main-proof}.

\begin{proof}[Proof of Theorem~\ref{thm:main}]
By Theorem~\ref{thm:blass}, $\MC$ holds. Let $X$ be an arbitrary set.
By Theorem~\ref{prop:embedding}, there exist a well-orderable
field $K$ of characteristic zero, a normal separable algebraic
extension $L/K$, and an injection $X\to L$.
Fix a well-ordering of $K$. Every intermediate field $F$ containing
$K$ has characteristic zero, so the hypothesis of
Theorem~\ref{thm:main} gives a basis of the $F$-vector space $L$.
Theorem~\ref{thm:algebraic} therefore gives a well-ordering of $L$,
which induces a well-ordering of $X$. Since $X$ was arbitrary,
$\AC$ follows.
\end{proof}

We now work backward from the conclusion of
Theorem~\ref{thm:algebraic}. If we can reach a well-ordered
intermediate field $F$ such that $L/F$ is finite, then a finite basis
gives a well-ordering of $L$ through its coordinates over $F$.
It is therefore enough to have a procedure that extends the current
well-ordered field, with the guarantee that whenever the procedure
adds nothing, the remaining extension is finite. The following theorem
provides precisely this step.

\begin{samepage}
\begin{overviewtheorem}\label{lem:adjoin}
Let $L/F$ be a normal separable algebraic extension, let $\prec$ be a
well-ordering of $F$, and let $\Ccal$ be a finite nonempty collection
of $F$-bases of $L$. From these data one can define an intermediate
field $F^+$, with $F\subseteq F^+\subseteq L$, and a well-ordering
$\prec^+$ of $F^+$ such that $(F,\prec)$ is an initial segment of
$(F^+,\prec^+)$ and
\[
 F^+=F\quad\Longrightarrow\quad [L:F]<\omega.
\]
\end{overviewtheorem}
\end{samepage}

Here the extension of the well-ordering is determined by the given
data. To prove Theorem~\ref{thm:algebraic} using this construction, we first apply
$\MC$ to the sets of bases over all well-orderable intermediate
fields, obtaining a finite nonempty family $\Ccal_F$ for each such
$F$. With these families fixed, we start at $K$ and repeatedly apply
the theorem. At limit stages we take unions. The requirement that each
order extend the previous one as an initial segment ensures that
these unions are still well-ordered.

The process must reach a stage at which the field no longer grows.
Indeed, let $\kappa=h(\Pow(L))$ be the Hartogs number of $\Pow(L)$,
the least ordinal that does not inject into the set of subsets of
$L$ (see \cite[Lemma~3.4, p.~29]{Jech2003}). A strictly increasing chain of intermediate fields of length
$\kappa$ would contradict its definition. At a stationary stage,
Theorem~\ref{lem:adjoin} gives a finite remaining extension, so the
finite-basis argument completes the well-ordering of $L$. This is
the proof of Theorem~\ref{thm:algebraic} from
Theorem~\ref{lem:adjoin}. Its formal details appear in
Section~\ref{sec:extension}.

It remains to explain the enlargement procedure in
Theorem~\ref{lem:adjoin}. Its definition is guided by a criterion for
the finiteness of $L/F$. To state that criterion, let
$\Ecal(L,F)$ denote the family of intermediate fields $E$ for which
$E/F$ is finite, normal and separable. If $B$ is an $F$-basis of $L$
and $E\in\Ecal(L,F)$, then $B\cap E$ is finite, since it is linearly
independent in the finite-dimensional $F$-vector space $E$. Put
\[
 q_{B,E}(x)=\prod_{b\in B\cap E}(x-b)\in E[x],
 \qquad
 P_{\Ccal,E}(z,x)=\prod_{B\in\Ccal}
                 \bigl(z-q_{B,E}(x)\bigr)\in E[z,x],
\]
where $\Ccal$ is a finite nonempty collection of $F$-bases of $L$ and
an empty product is $1$. The first polynomial records the finite
intersection with one basis. The second records the collection of
these polynomials without selecting a basis from $\Ccal$. Here $z$
is a second indeterminate, independent of $x$.

\begin{overviewtheorem}\label{lem:finite-bases}
Let $L/F$ be a normal separable algebraic extension and let $\Ccal$
be a finite nonempty family of $F$-bases of $L$. Suppose that, for every $E\in\Ecal(L,F)$,
\begin{equation*}
 P_{\Ccal,E}(z,x)\in F[z,x].
\end{equation*}
Then $L/F$ is finite.
\end{overviewtheorem}

This criterion suggests which elements should be adjoined in
Theorem~\ref{lem:adjoin}. Writing $\coeff(P)$ for the finite set of
coefficients of a polynomial $P$, we put
\[
 F^+=F\left(\bigcup_{E\in\Ecal(L,F)}\coeff(P_{\Ccal,E})\right).
\]
If $F^+=F$, all the polynomials $P_{\Ccal,E}$ have their coefficients
in $F$, and Theorem~\ref{lem:finite-bases} gives $[L:F]<\omega$.
To complete the enlargement step, we must also extend the given
well-ordering to $F^+$. Each $E\in\Ecal(L,F)$ is generated by the
roots of a polynomial over $F$. We can therefore index the
coefficients used above by polynomials in $F[x]$ and by pairs of
natural numbers recording the powers of $z$ and $x$. Since $F$ is
well-ordered, this is a well-ordered index set. Ordering the formal
expressions in these generators then gives the required extension
of the order, as shown in Section~\ref{sec:extension}.

The proof of Theorem~\ref{lem:finite-bases} is where Galois theory
enters. If $P_{\Ccal,E}$ has coefficients in $F$, every automorphism
of $E$ over $F$ sends a set $B\cap E$, with $B\in\Ccal$, to another
set $B'\cap E$ with $B'\in\Ccal$. This bounds the number of possible
conjugates of the coefficients of the polynomials $q_{B,E}$.
For one fixed basis in $\Ccal$, the resulting degree bound, together
with linear independence, forces that basis to be finite.
Section~\ref{sec:finite} proves this criterion using the elementary
Galois-theoretic facts established in Section~\ref{sec:galois}.

\section{Set-theoretic preliminaries}\label{sec:preliminaries}

For undefined set-theoretic notions, we refer the reader to \cite{Kunen2009} and \cite{Kunen2011}.

The Axiom of Multiple Choice $\MC$ is the statement that
for every family of nonempty sets $(A_i)_{i\in I}$ there exists
a function $f$ whose domain is $I$ such that for every $i\in I$,
$f(i)$ is a nonempty finite subset of $A_i$.

For a set $Y$, let $h(Y)$ denote the Hartogs number
of $Y$, that is, the least
ordinal that does not inject into $Y$. Hartogs's theorem says that such a
number exists (see, e.g., \cite[\S I.11]{Kunen2009}).

We shall use that a finite union of finite sets is finite and that
choice functions for finite collections of nonempty sets exist in $\ZFm$.

The following is known as L\'evy's lemma \cite{Levy1962}.
See also
\cite[Lemma~1, p.~105]{Tachtsis2019}.

\begin{lemma}[L\'evy]\label{lem:partition}
In $\ZFm$, $\MC$ is equivalent to the statement that
every set has a partition into finite nonempty sets
indexed by an ordinal.

That is, for every set $X$ there exists an ordinal $\alpha$ and a family $(X_\beta)_{\beta<\alpha}$ of nonempty finite subsets of $X$ such that for every two distinct $\beta,\gamma<\alpha$, $X_\beta\cap X_\gamma=\varnothing$ and $\bigcup_{\beta<\alpha}X_\beta=X$.
\end{lemma}

\section{Elementary Galois theory}\label{sec:galois}

In this section, we recall basic known facts on Galois theory for three purposes: to make this paper self-contained, as they will be used in the upcoming sections; to guarantee that the proofs of these facts do not require the Axiom of Choice; and to establish our notation.
For undefined notions in field theory, we refer to \cite{Rotman1998}.

We work solely in $\ZFm$ in this section.
\vspace{1em}

By $E/F$ we denote the fact that $E$ is a field extension of $F$.

If $E/F$ is a field extension and $A\subseteq E$, then $F(A)$ denotes the smallest subfield of $E$ containing both $F$ and $A$.
Similarly, if $\vec a=(a_1,\ldots,a_r)$ is a finite tuple of elements of $E$, then $F(\vec a)=F(\{a_1,\ldots,a_r\})$.
An extension $E/F$ is simple if $E=F(a)$ for some $a\in E$.
An extension $E/F$ is finitely generated if $E=F(\vec a)$ for some finite tuple $\vec a$ of elements of $E$.

As usual, we say that $E/F$ is a finite extension if $E$ is a finite-dimensional vector space over $F$.
In this case, the degree of a finite extension $E/F$ is the dimension of $E$ as an $F$-vector space, denoted $[E:F]$.
Recall, from basic linear algebra, that the Axiom of Choice is not needed to define dimension in finite-dimensional vector spaces.

If $E/F$ and $F/K$ are field extensions, assume that $\mathcal B$ is a basis of $E$ over $F$ and $\mathcal C$ is a basis of $F$ over $K$. Then the set $\{bc:b\in\mathcal B, c\in\mathcal C\}$ is a basis of $E$ over $K$ and the displayed elements are all distinct.
In particular, if $E/F$ and $F/K$ are finite extensions, then $E/K$ is a finite extension and $[E:K]=[E:F][F:K]$.
This last equality is known as the tower formula.

Let $E/F$ be a field extension and let $a \in E$.
We say that $a$ is \emph{algebraic} over $F$ if there exists a nonzero polynomial $p \in F[x]$ such that $p(a) = 0$.
In that case, the minimal polynomial of $a$ over $F$ is the unique monic irreducible polynomial $p \in F[x]$ such that $p(a) = 0$.
From basic field theory, it follows that $[F(a):F]=d$, where $d$ is the degree of $p$, and that $\{a^i:i<d\}$ is a basis of $F(a)$ over $F$.
In particular, every element of $F(a)$ can be written uniquely as $\sum_{i<d}c_i a^i$ for some $(c_i:i<d)\in F^d$.

A field extension $E/F$ is \emph{algebraic} if every element of $E$ is algebraic over $F$.
Every finite extension is algebraic.
By induction and the tower formula, if $ E=F(a_1,\ldots,a_n)
$ where  each $a_i$ is algebraic over $F$,
then $E/F$ is a finite extension, and hence algebraic.

Recall that an algebraic extension $E/F$ is:
\begin{itemize}
       \item \emph{normal} if every irreducible polynomial over
$F$ that has a root in $E$ splits over $E$;
       \item \emph{separable} if the minimal polynomial of every element of $E$ over $F$ has distinct roots.
\end{itemize}

Given a field $F$, we denote the ring of polynomials of $F$ in the variable $x$ by $F[x]$.

For a field extension $E/F$, a field $K$, and a field embedding $\tau:F\to K$, we use the notation
\[
 \Emb_{\tau}(E,K)
 =\{\sigma:E\hookrightarrow K:\sigma\text{ is a field embedding and }
       \sigma|_F=\tau\}.
\]
In case $\tau$ is the inclusion map from $F$ to $K$, we write $\Emb_F(E,K)$.
We call its elements $F$-embeddings of $E$ into $K$.

The \emph{Galois group} of $E/F$ is $\Gal(E/F)=\{\sigma\in\Emb_F(E,E):\sigma\text{ is an automorphism}\}$, which is a group under composition.
If $E/F$ is a finite extension, then $\Emb_F(E,E)=\Gal(E/F)$, since every injective $F$-linear map $E\to E$ is surjective.

For a field embedding $\tau:K\to E$ and a polynomial
$p=\sum_{i=0}^n a_i x^i\in K[x]$, write $\tau(p)=\sum_{i=0}^n\tau(a_i)x^i\in E[x]$.
This induces an extension of $\tau$ to a ring homomorphism $\tau:K[x]\to E[x]$.

\begin{lemma}\label{lem:simple-embeddings}
Let $H=K(a)$ be a simple algebraic extension, let $E$ be a field,
and let $\tau:K\to E$ be a field embedding.
If $p\in K[x]$ is the monic minimal polynomial of $a$ over $K$,
then the map from $\Emb_{\tau}(H,E)$ into $\{b\in E:\tau(p)(b)=0\}$ defined by
\[
 \sigma\longmapsto\sigma(a),
\]
is a bijection.
\end{lemma}

\begin{proof}
If $\sigma\in\Emb_{\tau}(H,E)$, write
$p=\sum_{i=0}^n c_i x^i$. Since $\sigma|_K=\tau$, we have
$
\tau(p)(\sigma(a))
=\sum_{i=0}^n\tau(c_i)\sigma(a)^i
=\sum_{i=0}^n\sigma(c_i)\sigma(a)^i
=\sigma\left(\sum_{i=0}^n c_i a^i\right)
=\sigma(p(a))=0.
$
Thus, the displayed map is well-defined.
Write $f(\sigma)=\sigma(a)$, so $f:\Emb_{\tau}(H,E)\to\{b\in E:\tau(p)(b)=0\}$ is a function.

Let $\mathrm{ev}_a:K[x]\to H$ be the evaluation map $q\mapsto q(a)$.
It is a surjective ring homomorphism with kernel $(p)$, so it induces
a unique isomorphism
$\overline{\mathrm{ev}}_a:K[x]/(p)\to H$ such that
$\overline{\mathrm{ev}}_a\circ\pi=\mathrm{ev}_a$, where
$\pi:K[x]\to K[x]/(p)$ is the canonical projection.

We will use $\overline{\mathrm{ev}}_a$ below to define an inverse mapping of $f$.

Fix a root $b\in E$ of $\tau(p)$ and define the ring homomorphism
$g_b:K[x]\to E$ by $g_b(q)=\tau(q)(b)$.
Since $b$ is a root of $\tau(p)$, we have $(p)\subseteq\ker g_b$.
As $p$ is irreducible, $(p)$ is maximal; moreover, $\ker g_b$ is
proper because $g_b(1)=1\neq 0$. Hence $\ker g_b=(p)$, and there is
a unique embedding $\bar g_b:K[x]/(p)\to E$ such that
$\bar g_b\circ\pi=g_b$. Define
\[
 \sigma_b=\bar g_b\circ\overline{\mathrm{ev}}_a^{-1}:H\to E.
\]
Since $\bar g_b$ is an embedding and $\overline{\mathrm{ev}}_a^{-1}$ is
an isomorphism, $\sigma_b$ is an embedding. Moreover, we have
$\sigma_b\circ\mathrm{ev}_a=\bar g_b\circ\overline{\mathrm{ev}}_a^{-1}
\circ\mathrm{ev}_a=\bar g_b\circ\overline{\mathrm{ev}}_a^{-1}
\circ\overline{\mathrm{ev}}_a\circ\pi=\bar g_b\circ\pi=g_b$.
Applying this identity to constant polynomials and to $x$ shows that
$\sigma_b(a)=b$ and $\sigma_b|_K=\tau$, so $\sigma_b\in\Emb_{\tau}(H,E)$.
Define $g:\{b\in E:\tau(p)(b)=0\}\to\Emb_{\tau}(H,E)$ by $g(b)=\sigma_b$.

Now, notice that given $b \in \{b\in E:\tau(p)(b)=0\}$, we have $f(g(b))=f(\sigma_b)=\sigma_b(a)=b$.
Conversely, given $\sigma \in \Emb_{\tau}(H,E)$, we have
$g(f(\sigma))=g(\sigma(a))=\sigma_{\sigma(a)}$,
and $\sigma_{\sigma(a)}=\sigma$, since these homomorphisms agree on $K$ and $a$.
Thus $g(f(\sigma))=\sigma$.
\end{proof}

\begin{lemma}\label{lem:count-embeddings}
Let $E/F$ be a finite normal separable extension, and let $\vec a$
be a finite tuple of elements of $E$.
Then $F(\vec a)/F$ is a finite extension and
\[
 |\Emb_F(F(\vec a),E)|=[F(\vec a):F].
\]
\end{lemma}

\begin{proof}
Write $\vec a=(a_1,\ldots,a_r)$ and $H=F(\vec a)$.
We prove the assertion by induction on $r$.
For $r=0$, we have $H=F$, so $\Emb_F(H,E)=\{\mathrm{id}_F\}$ and $[H:F]=[F:F]=1$, as desired.

For the induction step, assume the assertion holds for tuples of length
$r$. We show it holds for 
tuples of length $r+1$.
Put $K=F(a_1,\ldots,a_r)$ and $H=K(a_{r+1})$.
By the induction hypothesis, $K/F$ is a finite extension, and $H/K$ is a finite extension since it is simple and algebraic.
Thus, $H/F$ is a finite extension and $[H:F]=[K:F][H:K]$.
Let $p$ be the minimal polynomial of $a_{r+1}$ over $K$.
Then $d=[H:K]$ is the degree of $p$.

Fix $\tau\in\Emb_F(K,E)$. The polynomial $p$ divides the
minimal polynomial $q$ of $a_{r+1}$ over $F$. Thus, $\tau(p)\mid q$, since $\tau$ fixes $F$.
By normality and separability, $q$ splits into distinct linear factors
in $E$, so $\tau(p)$ has exactly $d$ roots in $E$.
By Lemma~\ref{lem:simple-embeddings}, $|\Emb_{\tau}(H,E)|=d$.

Consider the restriction map $\rho:\Emb_F(H,E)\to\Emb_F(K,E)$ defined by
\[
 \rho(\sigma)=\sigma|_K.
\]
For every $\tau\in\Emb_F(K,E)$, its fiber is
$\rho^{-1}(\{\tau\})=\Emb_{\tau}(H,E)$ and therefore has $d$ elements.
These fibers are pairwise disjoint and their union is $\Emb_F(H,E)$.
By the induction hypothesis, $\Emb_F(K,E)$ has $[K:F]$ elements.
Consequently,
\[
 |\Emb_F(H,E)|
 =\sum_{\tau\in\Emb_F(K,E)}|\rho^{-1}(\{\tau\})|
 =[K:F]d=[H:F].
\]
\end{proof}

Notice that if $E/H$ and $H/F$ are field extensions and $E/F$ is a finite extension, then $E/H$ is a finite extension as well:
every linearly independent set of $E$ over $H$ is also
linearly independent over $F$, so it is finite, thus a
linearly independent set of $E$ over $H$ of maximum length is a basis of $E$ over $H$.

Again, if $E/H$ and $H/F$ are field extensions and $E/F$ normal, so is $E/H$: if $p\in H[x]$ is irreducible over $H$ and has a root in $E$, let $a \in E$ be such a root.
Then $p$ generates the ideal $(p)$ of the polynomials in $H[x]$ that vanish at $a$.
The minimal polynomial $q$ of $a$ over $F$ is irreducible over $F$ and has a root in $E$, so it splits into linear factors in $E$.
Moreover, in $H[x]$, $q\in (p)$ as $q(a)=0$, so $p$ divides $q$ in $H[x]$.
As $E[x]$ is a unique factorization domain, $p$ also splits into linear factors in $E$.

Finally, if $E/H$ and $H/F$ are field extensions and $E/F$ separable, so is $E/H$: if $a \in E$, let $p$ be its minimal polynomial over $F$ and $q$ its minimal polynomial over $H$.
As before, $q$ divides $p$ in $H[x]$, so $q$ is also separable.

In the notation below, $\sigma(\vec a)=\sigma(a_1,\ldots,a_r)=(\sigma(a_1),\ldots,\sigma(a_r))$ whenever $\vec a=(a_1,\ldots,a_r)$ is a finite tuple of elements of $E$ and $\sigma$ is a field embedding of $E$ into some field.

\begin{lemma}\label{lem:orbit}
Let $E/F$ be a finite normal separable extension, and let $\vec a$
be a finite tuple of elements of $E$. Then
\[
 [F(\vec a):F]
 =\bigl|\{\sigma(\vec a):\sigma\in\Gal(E/F)\}\bigr|.
\]
\end{lemma}

\begin{proof}
Put $H=F(\vec a)$. The extension $E/H$ is finite, normal and
separable.
Fix a finite tuple $\vec e$ generating $E$ over $F$, which also generates
$E$ over $H$. Applying Lemma~\ref{lem:count-embeddings} to $\vec e$ over
$F$ gives $
 |\Emb_F(E,E)|=|\Emb_F(F(\vec e),E)|=[F(\vec e):F]=[E:F]
$. Likewise, applying the lemma to the finite normal separable extension
$E/H$ and to $\vec e$ gives $
|\Emb_H(E,E)|= |\Emb_H(H(\vec e),E)|=[H(\vec e):H]=[E:H]$. Using that embeddings of a finite extension into itself are
automorphisms, we obtain
\[
 |\Gal(E/F)|=[E:F],\qquad |\Gal(E/H)|=[E:H].
\]

Put $G=\Gal(E/F)$ and $S=\Gal(E/H)$. For $\sigma,\tau\in G$,
\[
 \sigma(\vec a)=\tau(\vec a)
 \iff \tau^{-1}\sigma\in S
 \iff \sigma S=\tau S.
\]
Indeed, the first equivalence holds because $\tau^{-1}\sigma$ fixes
every entry of $\vec a$ exactly when it fixes $H=F(\vec a)$; the second
is the equality criterion for left cosets. Thus the map from the set of
left cosets
\[
 \{\sigma S:\sigma\in G\}\longrightarrow\{\sigma(\vec a):\sigma\in G\},
 \qquad \sigma S\longmapsto\sigma(\vec a),
\]
is well-defined and bijective.
Each coset has $|S|$ elements, so finite counting and the tower
formula give
\[
 \bigl|\{\sigma(\vec a):\sigma\in G\}\bigr|
 =\frac{|G|}{|S|}
 =\frac{[E:F]}{[E:H]}=[H:F].
\]
\end{proof}

For a nonzero polynomial $p\in F[x]$, an extension $E/F$ is a
\emph{splitting field} of $p$ over $F$ if $p$ splits into linear
factors in $E[x]$ and $E=F(R)$, where $R=\{a\in E:p(a)=0\}$.

\begin{lemma}\label{lem:splitting-normal}
An extension $E/F$ is finite and normal if and only if it is a
splitting field of a nonzero polynomial over $F$.
\end{lemma}

\begin{proof}
Suppose first that $E/F$ is finite and normal. Take a finite
$F$-basis $b_1,\ldots,b_n$ of $E$, and let $p_i\in F[x]$ be the
minimal polynomial of $b_i$ over $F$. Each $p_i$ splits in $E$,
so $p=\prod_{i=1}^n p_i$ splits in $E$. Let $R$ be the set of roots
of $p$ in $E$.
As $b_i\in R$ for each $i$, we have $E=F(R)$. Thus $E$ is a splitting
field of $p$.

Conversely, suppose $E$ is a splitting field of $p\in F[x]$.
Its finite set of roots consists of algebraic elements, so $E/F$
is finite. Let $q\in F[x]$ be monic irreducible with a root $a\in E$.
There is a finite extension $M/E$ in which $q$ splits: successively
adjoin a root of a nonconstant irreducible factor of the remaining
polynomial, using the quotient by that factor, and divide out the
resulting linear factor. Induction on the remaining degree terminates
after at most $\deg q$ steps. This uses only finitely many choices.

Fix a root $b\in M$ of $q$. By Lemma~\ref{lem:simple-embeddings}, applied
with $\tau:F\to M$ the inclusion map,
there exists an $F$-embedding $\tau_0:F(a)\to M$ sending $a$ to $b$.
Write $K_0=F(a)$.

List the roots of $p$ in $E$ as $u_1,\ldots,u_n$.
They generate $E$ over $F$, and hence over $F(a)$.
Put $K_j=F(a,u_1,\ldots,u_j)$ for $0\le j\le n$.
We extend $\tau_0$ to an $F$-embedding $\tau_j:K_j\to M$ inductively.
Given $\tau_j$ with $j<n$, let $m\in K_j[x]$ be the minimal
polynomial of $u_{j+1}$ over $K_j$. Since $p(u_{j+1})=0$,
we have $m\mid p$, and thus $\tau_j(m)\mid p$, since $\tau_j$ fixes $F$.
The polynomial $p$ splits in $M[x]$, and $\tau_j(m)$ has the same
positive degree as $m$. Hence $\tau_j(m)$ also splits in $M[x]$
and has a root $v\in M$. Apply Lemma~\ref{lem:simple-embeddings} for
$K_j$, $u_{j+1}$, $K_{j+1}=K_j(u_{j+1})$, $M$, and
$\tau_j$. The minimal polynomial of $u_{j+1}$ over $K_j$ is $m$, and
$v$ is a root of $\tau_j(m)$, so the lemma gives an embedding
$\tau_{j+1}\in\Emb_{\tau_j}(K_{j+1},M)$ with
$\tau_{j+1}(u_{j+1})=v$. In particular, $\tau_{j+1}$ extends $\tau_j$.
Since $K_n=E$, this gives an $F$-embedding $\sigma=\tau_n:E\to M$
with $\sigma(a)=\tau_{0}(a)=b$.

Each $\sigma(u_i)$ is a root of $p$, and all roots of $p$ in $M$
already lie in $E$, since $p$ splits in $E$.
Hence $\sigma(E)\subseteq E$, and $b\in E$.
Since $b$ was an arbitrary root of $q$ in $M$, the polynomial $q$
splits in $E$. Thus $E/F$ is normal.
\end{proof}

Fix a normal separable algebraic extension $L/F$, and write
\[
 \Ecal(L,F)=\{E\in \mathcal P(L):F\subseteq E\subseteq L,
       \ E/F\text{ is finite, normal and separable}\}.
\]

\begin{lemma}\label{lem:finite-subextensions}
Let $L/F$ be a normal separable algebraic extension.
The family $\Ecal(L,F)$ is closed under finite composita and
\[
 \bigcup\Ecal(L,F)=L.
\]
\end{lemma}

\begin{proof}
The empty compositum is $F$, which belongs to $\Ecal(L,F)$.
For $E_1,\ldots,E_r\in\Ecal(L,F)$, let
$E=F(E_1\cup\cdots\cup E_r)$ be their compositum inside $L$.
By Lemma~\ref{lem:splitting-normal}, each $E_i$ is a splitting
field of some nonzero $p_i\in F[x]$. The product $p_1\cdots p_r$
splits in $E$, and its roots generate $E$.
Thus $E$ is a splitting field of this product, so
Lemma~\ref{lem:splitting-normal} shows that $E/F$ is finite and normal.
It is separable because $E\subseteq L$ and $L/F$ is separable.
Hence $E\in\Ecal(L,F)$.

Now we show that $\bigcup\Ecal(L,F)=L$.
Let $a\in L$, let $p\in F[x]$ be its minimal polynomial.
Since $L/F$ is normal, $p$ splits in $L$.
Put $R=\{b\in L:p(b)=0\}$. Then $F(R)$ is a splitting field
of $p$, so $F(R)/F$ is finite and normal by
Lemma~\ref{lem:splitting-normal} and $F(R)/F$ is separable as it is a subextension
of $L/F$. Since $a\in F(R)\in\Ecal(L,F)$, every element of $L$
belongs to a member of $\Ecal(L,F)$. The reverse inclusion is trivial.
\end{proof}

If $E/F$ is an algebraic field extension and $F$ has characteristic zero, then $E/F$ is separable.
We give a direct proof for the normal case to make this paper self-contained.
If $p(x)=\sum_{i=0}^n c_i x^i\in F[x]$ is a polynomial, its formal derivative is
\[
 p'(x)=\sum_{i=1}^n i c_i x^{i-1}\in F[x].
\]
The formal derivative is an $F$-linear map from $F[x]$ to itself, and it satisfies the Leibniz rule, $(pq)'=p'q+pq'$ for all $p,q\in F[x]$.

\begin{lemma}\label{lem:normal-separable}
Let $E/F$ be a normal algebraic extension, and assume that $F$ has characteristic zero. Then $E/F$ is separable.
\end{lemma}
\begin{proof}
       Let $a\in E$ and let $p\in F[x]$ be its minimal polynomial over $F$. Since $E/F$ is normal, $p$ splits in $E$.
       Write $p(x)=\prod_{i=1}^n (x-a_i)^{m_i}$, where $a_1,\ldots,a_n$ are distinct roots of $p$ in $E$ and $m_i\ge 1$ is the multiplicity of $a_i$.
       If for some $j$ we have $m_j>1$, write $p(x)=(x-a_j)^{m_j}q(x)$, where $q(x)=\prod_{i\ne j}(x-a_i)^{m_i}$. By the Leibniz rule,
       \[
       p'(x)=m_j(x-a_j)^{m_j-1}q(x)+(x-a_j)^{m_j}q'(x).
       \]
       Hence $(x-a_j)^{m_j-1}$ divides $p'(x)$, and therefore $p'(a_j)=0$.
       Since $p$ is irreducible over $F$ and $p(a_j)=0$, it is also the minimal polynomial of $a_j$ over $F$.
       Moreover, as $m_j>1$ and $F$ has characteristic zero, $1\leq\deg p'=\deg p-1<\deg p$, contradicting its minimality.
\end{proof}

\section{Finite families of bases}\label{sec:finite}

Given a field extension $L/F$, $E\in\Ecal(L,F)$, and an $F$-basis $B$ of $L$, the intersection $B\cap E$ is finite as it is a set of linearly independent elements of the finite-dimensional $F$-vector space $E$.
In this context, we define:
\[
       q_{B,E}(x)=\prod_{b\in B\cap E}(x-b)\in E[x].
\]
In the notation above (and whenever needed), the empty product is defined to be $1$.

Moreover, if $\Ccal$ is a finite nonempty collection of $F$-bases of $L$, we define
\[
 P_{\Ccal,E}(z,x)=\prod_{B\in\Ccal}\bigl(z-q_{B,E}(x)\bigr)\in E[z,x].
\]

\begin{lemma}\label{lem:basis-polynomials}
Let $L/F$ be a normal separable algebraic extension, let
$E\in\Ecal(L,F)$, and let $\Ccal$ be a finite nonempty collection
of $F$-bases of $L$. Suppose that $P_{\Ccal,E}\in F[z,x]$.

Then, for every $B\in\Ccal$ and every $\sigma\in\Gal(E/F)$,
there exists $B'\in\Ccal$ such that
\[
 \sigma[B\cap E]=B'\cap E.
\]
\end{lemma}

\begin{proof}
Enumerate $\Ccal$ without repetitions as $(B_i)_{i<m}$.
Suppose $P_{\Ccal,E}\in F[z,x]$, and fix $i<m$
and $\sigma\in\Gal(E/F)$.
Then $\sigma(P_{\Ccal,E})=P_{\Ccal,E}$.
Since $q_{B_i,E}$ is a root of this polynomial in the variable $z$,
$\sigma(q_{B_i,E})$ is a root of
$\sigma(P_{\Ccal,E})=P_{\Ccal,E}$ in the variable $z$.
Hence,
\[
 \prod_{j<m}
       \bigl(\sigma(q_{B_i,E})-q_{B_j,E}\bigr)=0
       \quad\text{in }E[x].
\]
As $E[x]$ is an integral domain, there exists $j<m$
such that $\sigma(q_{B_i,E})=q_{B_j,E}$.
The set of roots of $\sigma(q_{B_i,E})$ is
$\sigma[B_i\cap E]$, and the set of roots of $q_{B_j,E}$
is $B_j\cap E$.
\end{proof}

\setcounter{claim}{0}
\begin{proof}[Proof of Theorem~\ref{lem:finite-bases}]
Let $m=|\Ccal|$ and fix $B_*\in\Ccal$.

By Lemma~\ref{lem:basis-polynomials}, for every $E\in\Ecal(L,F)$ and
$\sigma\in\Gal(E/F)$ there exists $B'\in\Ccal$ with
$\sigma(B_*\cap E)=B'\cap E$.

\begin{claim}\label{claim:coeff-degree}
       For every $E_1,\ldots,E_r\in\Ecal(L,F)$, and every tuple $\vec a$ listing all the nonzero coefficients of the polynomials $q_{B_*,E_1},\ldots,q_{B_*,E_r}$, we have $[F(\vec a):F]\le m$.
\end{claim}
\begin{proof}[Proof of Claim]
       Let $E=F(E_1\cup\dots\cup E_r)\in\Ecal(L,F)$.
       Fix a tuple $\vec a$ as in the statement.

       Suppose that $\sigma,\tau\in\Gal(E/F)$ satisfy
       $\sigma[B_*\cap E]=\tau[B_*\cap E]$.
       Since $E_i/F$ is normal for each $i$, we have
       $\sigma(E_i)=\tau(E_i)=E_i$.
       Thus, for every $1\le i\le r$,
       \begin{align*}
       \sigma[B_*\cap E_i]
       &=\sigma[B_*\cap E\cap E_i]
        =\sigma[B_*\cap E]\cap\sigma(E_i)\\
       &=\sigma[B_*\cap E]\cap E_i
        =\tau[B_*\cap E]\cap E_i\\
       &=\tau[B_*\cap E]\cap\tau(E_i)
        =\tau[B_*\cap E\cap E_i]\\
       &=\tau[B_*\cap E_i].
       \end{align*}
       Therefore, for every $1\le i\le r$,
       \[
       \begin{aligned}
       \sigma(q_{B_*,E_i})
       &=\prod_{b\in B_*\cap E_i}(x-\sigma(b))
        =\prod_{c\in\sigma[B_*\cap E_i]}(x-c)\\
       &=\prod_{c\in\tau[B_*\cap E_i]}(x-c)
        =\prod_{b\in B_*\cap E_i}(x-\tau(b))
        =\tau(q_{B_*,E_i}).
       \end{aligned}
       \]
       
       Comparing coefficients, we see that $\sigma$ and $\tau$
       agree on every entry of $\vec a$, so
       $\sigma(\vec a)=\tau(\vec a)$.
       Thus we may define $\ell:\{\sigma[B_*\cap E]:\sigma\in\Gal(E/F)\}\to\{\sigma(\vec a):\sigma\in\Gal(E/F)\}$ by
       \[
       \ell(\sigma[B_*\cap E])=\sigma(\vec a).
       \]
       By Lemma~\ref{lem:basis-polynomials}, $\dom \ell\subseteq\{B\cap E:B\in\Ccal\}$, which has cardinality at most $m$, and $\ell$ is onto.
       Therefore, by Lemma~\ref{lem:orbit},
       \[
       [F(\vec a):F]
       =|\{\sigma(\vec a):\sigma\in\Gal(E/F)\}|\leq|\{B\cap E:B\in\Ccal\}|\le m.\qedhere
       \]
\end{proof}

Let $\coeff(q)$ denote the finite set of coefficients of a
polynomial $q$, and define
\begin{equation*}
S=\bigcup_{E\in\Ecal(L,F)}\coeff(q_{B_*,E}),\qquad
 D=F\left(S\right).
\end{equation*}

\begin{claim}\label{claim:D-degree} $[D:F]\le m$.
\end{claim}
\begin{proof}[Proof of Claim]
Let $d_1,\ldots,d_{m+1}\in D=F(S)$.
There exists a finite subset $S'\subseteq S$
such that $d_1,\ldots,d_{m+1}\in F(S')$.

Choose
$E_1,\ldots,E_r\in\Ecal(L,F)$ so that $S'\subseteq\bigcup_{i=1}^r\coeff(q_{B_*,E_i})$.

Let $\vec a$ be a tuple listing all the nonzero coefficients of the polynomials $q_{B_*,E_1},\ldots,q_{B_*,E_r}$.
Then $d_1,\ldots,d_{m+1}\in F(S')\subseteq F(\vec a)$ (as $0 \in F$).
As $[F(\vec a):F]\le m$ by Claim~\ref{claim:coeff-degree}, we conclude that $(d_1,\ldots,d_{m+1})$ is not $F$-linearly independent.

Thus every $F$-linearly independent tuple in $D$ has length at most $m$.
There is therefore a greatest possible length $r\le m$ of such a tuple.
Fix a linearly independent tuple of length $r$. It spans $D$, since any
element outside its span would extend it to a longer linearly independent
tuple. Hence it is a finite basis of $D$ over $F$, and $[D:F]=r\le m$.
\end{proof}
For each $b\in B_*$, let $\mu_b(x)\in D[x]$ be its monic minimal polynomial over $D$.
\begin{claim}\label{claim:roots-in-L}
       For every $b\in B_*$, all the roots of $\mu_b$ belong to $B_*$, and $\mu_b$ splits into distinct linear factors in $L[x]$.
\end{claim}
\begin{proof}[Proof of Claim]
Fix $b$ and let $E\in\Ecal(L,F)$ be such that $b\in E$.
By the definition of $D$,
$q_{B_*,E}$ belongs to $D[x]$.
As $q_{B_*,E}(b)=0$, it follows that $\mu_b(x)$ divides $q_{B_*,E}(x)$ in $D[x]$.
This implies that every root of $\mu_b$ is a root of $q_{B_*,E}$, and hence belongs to $B_*$.

Moreover, since $q_{B_*,E}$ splits into distinct linear factors in $L[x]$, so does $\mu_b$.
\end{proof}

Define in $B_*$ the equivalence relation $\sim$ given by
$b\sim c\Longleftrightarrow\mu_b=\mu_c$.
For each $b\in B_*$, let $C_b$ be its equivalence class under $\sim$. We claim that
$C_b=\{c\in L:\mu_b(c)=0\}$.
Indeed, if $c\in C_b$, then $\mu_c=\mu_b$, so $\mu_b(c)=0$.
Conversely, if $\mu_b(c)=0$, then $c\in B_*$ by
Claim~\ref{claim:roots-in-L}. Moreover, $\mu_c$ divides $\mu_b$, and
since $\mu_b$ is irreducible and both polynomials are monic, we have
$\mu_c=\mu_b$. Thus $c\in C_b$. In particular, every equivalence class
is finite and nonempty.

For each equivalence class $C$, define
\[
 S_C=\sum_{b\in C}b.
\]
We claim that $S_C\in D$. Indeed, fix $b\in C$.
Since the roots of $\mu_b$ are precisely the elements of $C$, are
distinct, and all have multiplicity one, Vi\`ete's formula shows that
$-S_C$ is a coefficient of $\mu_b$.
Hence $S_C\in D$.

Now we claim that the family $(S_C: C \in B_*/\sim)$ is
$F$-linearly independent.
Indeed, let $(C_i: i<n)$ be a finite family of pairwise distinct equivalence classes, and let $(a_i: i<n)$ be a family of elements of $F$ such that
\[
 \sum_{i< n} a_i S_{C_i}=0.
\]

Then
\[
 0=\sum_{i<n} a_i S_{C_i}
  =\sum_{i<n} a_i\sum_{b\in C_i}b
  =\sum_{i< n}\sum_{b\in C_i}a_ib.
\]
Since the classes are pairwise disjoint and subsets of the $F$-basis $B_*$, each $a_i$ must be zero.

Finally, by Claim~\ref{claim:D-degree}, an $F$-linearly independent family in
$D$ has at most $m$ elements. Hence there are at most $m$ equivalence
classes. Each class is finite, so their union $B_*$ is finite. Since
$B_*$ is an $F$-basis of $L$, the extension $L/F$ has finite degree.
\end{proof}

\section{Well-ordering an algebraic extension}\label{sec:extension}

In this section, we work in $\ZFm$, except in Theorem~\ref{thm:algebraic}, where we also assume $\MC$.

The main goal of this section is to prove Theorem~\ref{thm:algebraic}, which shows that, assuming $\MC$ and given a field extension $L/K$ that is normal, separable and algebraic, if certain intermediate extensions of $L/K$ have bases and if $K$ is well-orderable, then this well-ordering can be extended to a well-ordering of $L$.
The idea of the proof is to recursively define a well-ordering of an increasing chain of subfields of $L$ starting with $K$ and ending with $L$.
The ideas are somewhat adapted from the standard well-ordering of the constructible universe $\mathbf L$ (see, e.g., \cite[Definitions~II.6.18 and~II.6.19 and Theorem~II.6.20]{Kunen2011}).

We use basic first-order logic internalized within $\ZFm$.
Formally, it is constructed using Polish notation, so no parentheses are needed.
The definition and basic development of internalized first-order logic and the satisfaction relation do not depend on the Axiom of Foundation or the Axiom of Choice.
We refer the reader to \cite[\S\S II.4--II.8]{Kunen2009} for details.

By $\Var$, we denote a countably infinite set of variables fixed in advance.

Let $\Symbols$ denote the set $\Var\cup\{0,1,+,-,\cdot\}\cup\{=,\land, \vee, \rightarrow, \leftrightarrow, \neg, \exists, \forall\}$.
Fix once and for all an order $<_S$ of type $\omega$ of the set $\Symbols$.

Consider the first-order language of field theory with constants in $D$, $\Lcal_D=\{0,1,+,-,\cdot\}\cup D$, where $D$ is a set of constant symbols disjoint from $\Symbols$.
For this to hold in the future, we can assume that $\Symbols$ has no pairs, triples or natural numbers (for example, code the symbols by distinct $4$-tuples of natural numbers).

The set of all $\Lcal_D$-formulas is countable in $\ZFm$ when $D$ is countable.

If $\varphi$ is an $\Lcal_D$-formula and $k \in \omega$, an enumeration of the free variables of $\varphi$ in $k$ parameters is a nonempty finite sequence $\vec v=(v_0,\ldots,v_k)$ of distinct variables in $\Var$ for which the set of free variables of $\varphi$ is contained in $\{v_0,\ldots,v_{k}\}$.
Notice that the length of $\vec v$ is $k+1$, not $k$.
That is happening on purpose: $v_0$ is a variable that will not be thought of as a parameter, as will hopefully be clear in the sequel.

An $\Lcal_D$ \emph{formula with $k$ parameters} is a pair $(\varphi,\vec v)$ consisting of an $\Lcal_D$-formula $\varphi$ and an enumeration of the free variables of $\varphi$ in $k$ parameters. We denote it by $\varphi(\vec v)$.

Fix once and for all a well-ordering $<_{\Var}$ of type $\omega$ of
$\Var^{<\omega}$.

Let $D$ be disjoint from $\Symbols$, and let $<_D$ be a
well-ordering of $D$. Define $<'_D$ on $D\cup\Symbols$ by
\[
 <'_D=<_S\cup \{(d,s):d\in D, s\in \Symbols\}\cup <_D.
\]
For $s,t\in(\Symbols \cup D)^{<\omega}$, define
\[
 s<^*_D t\iff |s|<|t|\ \text{or}\
 \bigl(|s|=|t|\ \text{and}\ s<^{|s|}_{D}t\bigr),
\]
where $<^{n}_D$ is the lexicographic order on $(\Symbols \cup D)^{n}$ induced by $<'_D$.
For $\mathcal L_D$-formulas with parameters
$(\varphi,\vec v)$ and $(\psi,\vec w)$, define
\[
 (\varphi,\vec v)\mathrel{\lhd_D}(\psi,\vec w)
 \quad\Longleftrightarrow\quad
 \varphi<^*_D\psi
 \quad\text{or}\quad
 \bigl(\varphi=\psi\ \text{and}\ \vec v<_{\Var}\vec w\bigr).
\]

\begin{lemma}
The relation $\lhd_D$ is a well-ordering of the set of all
$\mathcal L_D$-formulas with parameters.
\end{lemma}
\begin{proof}
The order $<'_D$ is a well-ordering of $D\cup\Symbols$. Hence,
for every $n\in\omega$, the lexicographic order $<^n_D$ is a
well-ordering of $(D\cup\Symbols)^n$. Comparing first by length and
then lexicographically therefore makes $<^*_D$ a well-ordering
of $(D\cup\Symbols)^{<\omega}$. Its restriction to the set of all
$\mathcal L_D$-formulas is consequently a well-ordering.

Finally, $\lhd_D$ is the lexicographic order induced by the
well-orderings $<^*_D$ and $<_{\Var}$. Its restriction to the set of
$\mathcal L_D$-formulas with parameters is therefore a
well-ordering.
\end{proof}

Let $E$ be an $\Lcal_D$-structure, and let $\varphi(\vec v)$ be a
formula with $k$ parameters. For $a\in E$ and
$\vec b=(b_i)_{i<k}\in E^k$, we write $E\models\varphi(a,\vec b)$ if $E\models\varphi[\sigma]$, where $\sigma:\Var\to E$ is a variable assignment such that $\sigma(v_0)=a$ and $\sigma(v_{i+1})=b_i$ for every $i<k$.

\begin{definition}\label{def:formula-order}
An \emph{extension tuple} is a tuple $T=(L,F,\prec,D,<_D,\mathbf u)$ such that:
\begin{itemize}
    \item $L/F$ is a field extension,
    \item $\prec$ is a well-ordering of $F$,
    \item $D$ is a set disjoint from $\Symbols$,
    \item $\mathbf u=(u_d)_{d\in D}$ is a family in $L$,
    \item $<_D$ is a well-ordering of $D$.
\end{itemize}
In the notation above, we define:
\begin{enumerate}
    \item $E_T=F(u_d:d\in D)$, where $E_T$ is regarded as an $\Lcal_D$-structure with the interpretation of each $d\in D$ being $u_d$.
    \item For each $k\in\omega$, let $\prec_{\mathrm{lex}}^k$ be the
    lexicographic well-ordering of $F^k$ induced by $\prec$. Define
    $\sqsubset_T$ on the pairs $(\varphi(\vec v),\vec b)$, where
    $\varphi(\vec v)$ is an $\Lcal_D$-formula with $k$ parameters and
    $\vec b\in F^k$, by
    \[
    \begin{aligned}
    (\varphi(\vec v),\vec b)\mathrel{\sqsubset_T}
    (\psi(\vec w),\vec c)
    \quad\Longleftrightarrow\quad&
    \varphi(\vec v)\mathrel{\lhd_D}\psi(\vec w)\\
    &\text{or}\quad
    \bigl(\varphi(\vec v)=\psi(\vec w)
    \ \text{and}\ \vec b\mathrel{\prec_{\mathrm{lex}}^k}\vec c\bigr).
    \end{aligned}
    \]
    In the second clause, equality of the first coordinates ensures
    that both tuples belong to the same $F^k$. Thus $\sqsubset_T$ is
    the lexicographic well-ordering obtained by first comparing the
    formulas with parameters and then their corresponding tuples from
    $F^k$.
\end{enumerate}

Such a pair $(\varphi(\vec v),\vec b)$ is said to \emph{describe}
$x\in E_T$ if
\[
 \{a\in E_T:E_T\models\varphi(a,\vec b)\}=\{x\}.
\]
Every $c\in E_T$ admits such
a description as every element of $F(\mathbf u)$ satisfies the relation
\[
    c\cdot \sum_{i \in I} b_i\prod_{j \in J_i} u_j^{m_{i,j}} = \sum_{i \in I} b_i'\prod_{j \in J_i} u_j^{m_{i,j}}
\]
for some finite set $I$, some finite subsets $J_i$ of $D$ for each $i \in I$, exponents $m_{i,j}\in\omega$ for $i\in I$ and $j\in J_i$, and $b_i,b_i' \in F$ for each $i \in I$ with $\sum_{i \in I} b_i\prod_{j \in J_i} u_j^{m_{i,j}}\neq 0$.
In this case, this relation determines $c$.

For $x\in E_T$, put
\[
 c_T(x)=\min_{\sqsubset_T}
       \{(\varphi(\vec v),\vec b):
              (\varphi(\vec v),\vec b)\text{ describes }x\}.
\]
Define a relation $\prec_T^+$ on $E_T$ by
\begin{equation}
 x\prec_T^+y\ \Longleftrightarrow\quad
 \begin{aligned}[t]
 &(x,y\in F\ \text{and}\ x\prec y)\\
 &\quad\text{or}\ (x\in F\ \text{and}\ y\in E_T\setminus F)\\
 &\quad\text{or}\ (x,y\in E_T\setminus F\ \text{and}\ c_T(x)\sqsubset_T c_T(y)).
 \end{aligned}
\end{equation}
\end{definition}

The following is straightforward from the previous definition and left to the reader.

\begin{lemma}\label{lem:orderExtension}
Let $T=(L,F,\prec,D,<_D,\mathbf u)$ be an extension tuple.
Then $(E_T,\prec_T^+)$ is a well-order and $(F,\prec)$ is an initial segment of $(E_T,\prec_T^+)$.
\end{lemma}

For $P\in L[z,x]$, recall that $\coeff(P)$ denotes its finite set of coefficients
as a polynomial in two variables.

If $F$ is a field and $\prec$ is a well-ordering of $F$, we define $\prec_x$ to be the lexicographic well-ordering of $F[x]$ induced by $\prec$: to compare two distinct polynomials $p$ and $q$ in $F[x]$, first compare their degrees, and if they are equal, compare their dominant coefficients using $\prec$, and then continue comparing the coefficients of lower degree terms in decreasing order of degree until a difference is found.
Here we take $\deg(0)=-\infty$ and omitted coefficients to be zero.
That is, if $p(x)=\sum_{i=0}^n a_ix^i$ and $q(x)=\sum_{i=0}^m b_ix^i$, and $p\neq q$, let $\Delta(p,q)=\max\{i\in\omega:a_i\neq b_i\}$.
Then
\begin{align*}
    p\prec_x q\quad\Longleftrightarrow\quad
    \begin{cases}
        \deg(p)<\deg(q),\\
        \text{or } \deg(p)=\deg(q) \text{ and } a_n\prec b_n, \text{ where } n=\Delta(p,q).
    \end{cases}
\end{align*}

\begin{proof}[Proof of Theorem~\ref{lem:adjoin}]
Let $L/F$ be a normal separable algebraic extension, let $\prec$ be a
well-ordering of $F$ and let $\Ccal$ be a finite nonempty collection of
$F$-bases of $L$. Write
\[
 F^+=F\left(\bigcup_{E\in\Ecal(L,F)}\coeff(P_{\Ccal,E})\right).
\]
From these objects one can define a well-ordering of $F^+$ that has
$(F,\prec)$ as an initial segment.

More specifically, let
\begin{itemize}
    \item $\mathscr P_F=\{p\in F[x]:p\text{ is monic and splits into distinct linear factors in }L[x]\}\cup\{1\}$,
    \item $E_p=F(\{u\in L:p(u)=0\})$ for $p\in\mathscr P_F$,
    \item $c_{p,i,j}$ is the coefficient of $z^ix^j$ in $P_{\Ccal,E_p}(z,x)$ for
    $(p,i,j)\in\mathscr P_F\times\omega\times\omega$,
    \item $D=\mathscr P_F\times\omega\times\omega$
    is equipped with the lexicographic well-ordering $<_D$ induced by
    $\prec_x$ and the usual well-ordering of $\omega$,
    \item $\mathbf c=(c_{p,i,j})_{(p,i,j)\in D}$, and
    \item $T=(L,F,\prec,D,<_D,\mathbf c)$.
\end{itemize}
Then $T$ is an extension tuple, $E_T=F^+$, and
$\prec_T^+$ is a well-ordering of $F^+$ having
$(F,\prec)$ as an initial segment.

First, we show that $\Ecal(L,F)=\{E_p:p\in\mathscr P_F\}$.

For each $p\in\mathscr P_F$, the field $E_p$ is the splitting field of
$p$ over $F$. By Lemma~\ref{lem:splitting-normal}, $E_p/F$ is finite
and normal. It is also separable, since $E_p\subseteq L$. Thus
$E_p\in\Ecal(L,F)$.

Conversely, let $E\in\Ecal(L,F)$. By
Lemma~\ref{lem:splitting-normal}, $E$ is the splitting field of some
nonzero $q\in F[x]$. When $q$ is
constant, we have $p=1$ and $E=F$.
Otherwise, let $q_0,\ldots,q_{r-1}$ be its distinct monic
irreducible factors and put $p=\prod_{i<r}q_i$. Each $q_i$ splits in
$E$ by normality, and its roots are all distinct since they belong to the separable
extension $E/F$. Distinct $q_i$'s have no common root as they are irreducible and monic, so
$p\in\mathscr P_F$. The polynomials $p$ and $q$ have the same roots,
and these roots generate $E$ over $F$. Hence $E=E_p$.

Now let $S=\bigcup_{E\in\Ecal(L,F)}\coeff(P_{\Ccal,E})$.
It follows that $S\cup\{0\}=\{c_{p,i,j}:(p,i,j)\in D\}$, so $F^+=F(S)=F(\mathbf c)$, so that $E_T=F^+$.
By Lemma~\ref{lem:orderExtension},
$\prec_T^+$ is a well-ordering of $E_T=F^+$ having
$(F,\prec)$ as an initial segment.

If $F^+=F$, then $P_{\Ccal,E}\in F[z,x]$ for every
$E\in\Ecal(L,F)$. Theorem~\ref{lem:finite-bases} therefore gives
$[L:F]<\omega$, as required.
\end{proof}

We now prove Theorem~\ref{thm:algebraic}, using $\MC$.

Under the stated hypotheses, it extends a given well-ordering of a field $K$ to a normal separable algebraic extension $L/K$.
The idea is to recursively construct an increasing chain of well-ordered intermediate fields, starting with $K$ and applying Theorem~\ref{lem:adjoin} at each successor stage in an attempt to reach an intermediate field satisfying the hypotheses of Theorem~\ref{lem:finite-bases}.
The Hartogs number $h(\Pow(L))$ guarantees that this chain eventually stabilizes.
When the chain finally stabilizes, Theorem~\ref{lem:finite-bases} shows that the remaining extension has finite degree, so a final application of Lemma~\ref{lem:orderExtension} yields the desired well-ordering of \(L\).

\begin{proof}[Proof of Theorem~\ref{thm:algebraic}]
Let $\mathscr S=\{F:K\subseteq F\subseteq L,
                 \ F\text{ is a well-orderable subfield of }L\}$.

For each $F\in\mathscr S$, let $\Bas_F(L)=\{B\subseteq L:B\text{ is an }F\text{-basis of }L\}$.
By hypothesis, $\Bas_F(L)$ is nonempty for each $F\in\mathscr S$.
By $\MC$, there exists $(\Ccal_F)_{F\in\mathscr S}$ such that $\varnothing\ne\Ccal_F\subseteq\Bas_F(L)$ and $\Ccal_F$ is finite for each $F\in\mathscr S$.

Let $\kappa=h(\Pow(L))$, the Hartogs number of $\Pow(L)$.
We may assume that $L$ is infinite (otherwise the theorem is trivial), so $\kappa$ is an infinite cardinal.
Recursively define $(F_\alpha,\prec_\alpha)$ for $\alpha<\kappa$, where $F_\alpha$ is a subfield of $L$ well-ordered by $\prec_\alpha$, along with extension tuples $(T_\alpha)_{\alpha<\kappa}$, families $(D_\alpha)_{\alpha<\kappa}$, well-orderings $(<_\alpha)_{\alpha<\kappa}$ of $D_\alpha$, and families $(\mathbf c_\alpha)_{\alpha<\kappa}$ in $L$ such that, for every $\alpha<\kappa$,
\begin{enumerate}
    \item $(F_0,\prec_0)=(K,\prec)$,
    \item $T_\alpha=(L,F_\alpha,\prec_\alpha,D_\alpha,<_\alpha,\mathbf c_\alpha)$,
    \item $D_\alpha=\mathscr P_{F_\alpha}\times\omega\times\omega$,
    \item $<_\alpha$ is the lexicographic well-ordering of $D_\alpha$ induced by $\prec_{\alpha,x}$ and the usual well-ordering of $\omega$,
    \item $\mathbf c_\alpha=(c_{p,i,j})_{(p,i,j)\in D_\alpha}$, where $c_{p,i,j}$ is the coefficient of $z^ix^j$ in $P_{\Ccal_{F_\alpha},E_p}$ for $(p,i,j)\in D_\alpha$, and $E_p=F_\alpha(\{u\in L:p(u)=0\})$ for $p\in\mathscr P_{F_\alpha}$.
    \item $F_{\alpha+1}=E_{T_\alpha}=F_\alpha\left(\bigcup_{E\in\Ecal(L,F_\alpha)}\coeff(P_{\Ccal_{F_\alpha},E})\right)$,
    \item $\prec_{\alpha+1}=\prec_{T_\alpha}^+$.
    \item $(F_\beta,\prec_\beta)$ is an initial segment of $(F_{\alpha},\prec_{\alpha})$ for every $\beta<\alpha$.
    \item $F_\alpha=\bigcup_{\beta<\alpha}F_\beta$ and $\prec_\alpha=\bigcup_{\beta<\alpha}\prec_\beta$ for limit ordinals $\alpha<\kappa$.
\end{enumerate}
At each successor stage, $L/F_\alpha$ is again a normal separable
algebraic extension, since the minimal polynomial over $F_\alpha$ of
an element of $L$ divides its minimal polynomial over $K$. Thus
Theorem~\ref{lem:adjoin} applies. At a limit stage, the union of the
preceding fields is a field, and the union of their well-orderings is
a well-ordering because they form an increasing chain of initial
segments. Hence the recursion can be carried out as specified.

For every $\beta<\alpha<\kappa$, we have $F_\beta\subseteq F_\alpha$.
Suppose for a contradiction that for every $\alpha<\kappa$ we have $F_{\alpha+1}\ne F_\alpha$.
Then $(F_\alpha)_{\alpha<\kappa}$ is a strictly increasing chain of subfields of $L$ of length $\kappa$, yielding an injection of $\kappa$ into $\Pow(L)$, contrary to the definition of $\kappa$.

Let $\alpha$ be the first ordinal such that $F_{\alpha+1}=F_\alpha$, and write $F=F_\alpha$.
Then, by the definition of $F_{\alpha+1}$, all the coefficients of
$P_{\Ccal_F,E}$ lie in $F$ for every $E\in\Ecal(L,F)$.
Theorem~\ref{lem:finite-bases} then implies $[L:F]<\omega$.

The field $F$ is well-ordered by $\prec_\alpha$.
Take a finite basis $\mathbf b=(b_i)_{i<m}$ of $L/F$, and let
$D$ be a copy of $m$ disjoint from $\Symbols$ with its usual order $<_D$. Then $T=(L,F,\prec_\alpha,D,<_D,\mathbf b)$ is an
extension tuple with $E_T=L$. Lemma~\ref{lem:orderExtension} shows
that $\prec_T^+$ has $(F_\alpha,\prec_\alpha)$ as an
initial segment. Since $(K,\prec)$ is an initial segment of
$(F_\alpha,\prec_\alpha)$, it is also an initial segment of this
well-ordering of $L$.
\end{proof}

\section{Embedding a set into a field extension}\label{sec:main-proof}

\begin{proof}[Proof of Theorem~\ref{prop:embedding}]
For $X=\varnothing$, take $K=L=\Q$. Otherwise, by
Lemma~\ref{lem:partition}, fix a partition $(X_\alpha)_{\alpha<\lambda}$ of $X$ into nonempty finite blocks, and put $n_\alpha=|X_\alpha|$.
Let $\Q[t_x:x\in X]$ be the polynomial ring in distinct indeterminates $(t_x)_{x\in X}$ over $\Q$, and let $L$ be its field of fractions:
\[
 L=\Q(t_x:x\in X)
\]
The map $x\mapsto t_x$ is injective and $L$ has a natural embedding of $\Q$ into it.
We switch to it and maintain the notation $\Q$, so $\Q\subseteq L$.

For each $\alpha<\lambda$, define
\[
 p_\alpha(y)=\prod_{u\in X_\alpha}(y-t_u)
            =y^{n_\alpha}+\sum_{j<n_\alpha}c_{\alpha,j}y^j
            \in L[y],
\]
Here the coefficients are canonically indexed by the powers $y^j$ they
multiply, so this indexing requires no choice.
Let $K$ be the subfield of $L$ generated by all the coefficients
of these polynomials:
\[
 K=\Q(c_{\alpha,j}:\alpha<\lambda,\ j<n_\alpha).
\]
Thus each $p_\alpha$ belongs to $K[y]$, splits in $L$, and has distinct roots.
Put
\[
 D=\{(\alpha,j):\alpha<\lambda,\ j<n_\alpha\},
 \qquad \mathbf c=(c_{\alpha,j})_{(\alpha,j)\in D}.
\]
Let $<_D$ be its natural lexicographic well-ordering.
As $\Q$ is countable in $\ZFm$, fix a well-ordering $\prec_{\Q}$ of $\Q$,
and define an extension tuple $T=(L,\Q,\prec_{\Q},D,<_D,\mathbf c)$.
Notice that $E_T=K$, so Lemma~\ref{lem:orderExtension} gives a well-ordering
$\prec_K=\prec_T^+$ of $K$, which contains $(\Q,\prec_{\Q})$
as an initial segment.

For any finite
$J\subseteq\lambda$, the field
\[
 E_J=K(t_x:x\in\bigcup_{\alpha\in J}X_\alpha)
\]
is the splitting field inside $L$ of
$\prod_{\alpha\in J}p_\alpha$. By Lemma~\ref{lem:splitting-normal},
it is finite and normal over $K$, and it is separable by Lemma~\ref{lem:normal-separable}. The fields $E_J$ form a directed family
whose union is $L$, hence $L/K$ is normal, separable and algebraic.
Explicitly:
\begin{itemize}
    \item $L/K$ is algebraic: given any $a \in L$, there exists a finite $J\subseteq\lambda$ such that $a\in E_J$.
    As $E_J/K$ is algebraic, it follows that $a$ is algebraic over $K$.
    \item $L/K$ is normal: let $a \in L$ and let $p(y)\in K[y]$ be the minimal polynomial of $a$ over $K$. There exists a finite $J\subseteq\lambda$ such that $a\in E_J$. Since $E_J/K$ is normal, $p(y)$ splits in $E_J$, and hence it splits in $L$.
    \item $L/K$ is separable as $K$ has characteristic zero, by Lemma~\ref{lem:normal-separable}.
\end{itemize}
\end{proof}

The proof of Theorem~\ref{thm:main} in Section~\ref{sec:structure} also gives the following.

\begin{corollary}\label{cor:restricted}
In $\ZFm$, the following are equivalent:
\begin{enumerate}
\item $\AC$;
\item $\MC$+``every vector space over every well-orderable field of characteristic zero has a basis''.
\end{enumerate}
\end{corollary}

\section{Concluding remarks}\label{sec:conclusion}

Theorem~\ref{thm:main} shows that Foundation is unnecessary for the
equivalence between the existence of vector space bases and $\AC$,
settling the questions that appear in
\cites[p.~33]{Blass}{Karagila2017MO}[p.~423]{Morillon}[p.~121]{Philip2025}.

Blass~\cite[p.~33]{Blass} also asked whether, without Foundation, $\AC$
follows from the assertion that every linearly independent set in a
vector space can be extended to a basis. The following corollary
answers this affirmatively and completes the following collection of
equivalences. The equivalence of $\AC$ with the assertion that every
spanning subset of a vector space contains a basis was already proved
by Halpern~\cite{Halpern1966} without Foundation, even allowing atoms
(see also Blass~\cite[p.~31]{Blass}).

\begin{corollary}\label{cor:basis-principles}
In $\ZFm$, the following statements are equivalent:
\begin{enumerate}
\item $\AC$;
\item every vector space has a basis;
\item every vector space over a field of characteristic zero has a basis;
\item every linearly independent subset of a vector space is contained
in a basis of that space;
\item every spanning subset of a vector space contains a basis of
that space.
\end{enumerate}
\end{corollary}

\begin{proof}
Assume $(1)$. Given a linearly independent subset $I$ of a vector
space $V$, Zorn's lemma yields a maximal linearly independent subset
of $V$ containing $I$. Such a subset spans $V$ and is therefore a
basis, proving $(4)$. Similarly, given a spanning subset $S$ of $V$,
Zorn's lemma yields a maximal linearly independent subset $B$ of $S$.
Maximality implies that $S$ is contained in the span of $B$, so $B$
is a basis of $V$, proving $(5)$. These applications of Zorn's lemma
do not require Foundation.

For an arbitrary vector space $V$, taking $I=\varnothing$ in $(4)$,
or $S=V$ in $(5)$, yields a basis of $V$. Hence both $(4)$ and $(5)$
imply $(2)$. Clearly, $(2)$ implies $(3)$. Finally, $(3)\Rightarrow(1)$ follows from Theorem~\ref{thm:main}.
\end{proof}

The argument also works in set theory with atoms. We use
$\mathsf{ZFA}$ in the sense of Jech~\cite[\S4.1, pp.~44--45]{Jech1973},
including its axiom of Regularity, and write $\mathsf{ZFA}^{-}$ for
the same theory with that axiom omitted. The argument works already
in $\mathsf{ZFA}^{-}$.
Indeed, Blass's implication and L\'evy's partition lemma hold in that
setting. In Theorem~\ref{prop:embedding}, the elements of $X$
serve only as indices for indeterminates, so the construction applies
to sets containing atoms. The subsequent arguments use finite
field extensions, sets of bases and subfields, and recursion on an
ordinal bounded by a Hartogs number, and all these constructions are
available in $\mathsf{ZFA}^{-}$.
Thus Theorem~\ref{thm:main} and
Corollary~\ref{cor:basis-principles} hold in $\mathsf{ZFA}^{-}$, and hence in $\mathsf{ZFA}$, as well.

Halbeisen~\cite[Note~94]{Halbeisen2012} asks whether, in $\mathsf{ZFA}$,
$\AC$ follows from the assertion that every vector space has a basis,
or at least from the stronger assertion that every linearly independent
subset of a vector space extends to a basis of that space. Our results
answer both questions affirmatively.

\subsection*{Acknowledgements}
This study was financed, in part, by the S\~ao Paulo Research Foundation
(FAPESP), Brasil. Process Numbers \#25/07302-0 and \#25/09425-1.

\subsection*{Declaration of AI assistance}
OpenAI's Codex, using the GPT-6 Astra model,
was used to generate novel mathematical content and assist with the writing process.
The assistance also included grammatical suggestions and bibliographical references.
All the output and suggestions were reviewed, edited and/or rewritten by the authors.

\printbibliography

\end{document}